\documentclass[11 pt]{article}

\usepackage{latexsym, amsmath, amssymb, longtable, booktabs,amscd,microtype,booktabs,amsthm,fancyhdr,amsfonts}

\usepackage{etoolbox}
\usepackage[affil-it]{authblk}
\usepackage[square]{natbib}
\usepackage{enumerate}
\usepackage{tikz}
\usetikzlibrary{matrix,arrows,decorations.pathmorphing}
\usepackage{hyperref}
\bibpunct{[}{]}{,}{n}{}{;} 

\numberwithin{equation}{section}
\newtheorem{thm}{Theorem}[subsection]

\theoremstyle{definition}

\theoremstyle{remark}
\newtheorem{rmk}[thm]{\textbf{Remark}}

\theoremstyle{definition}
\newtheorem{dfn}{Definition}[subsection]

\theoremstyle{remark}

\theoremstyle{remark}

\makeatletter
\def\imod#1{\allowbreak\mkern10mu({\operator@font mod}\,\,#1)}

\makeatother

\title{\textbf{On the proof of the Prime Number Theorem using Order Estimates for the Chebyshev Theta Function}}
\author{SUBHAM DE}
\affil{Department of Mathematics}
\affil{Indian Institute of Technology, Delhi, India.}
\affil{Email: \textbf{mas227132@iitd.ac.in}}

\date{}

\begin{document}
	
	\maketitle
	\thispagestyle{empty}

\begin{abstract}
	In this paper, we shall study the stellar work of Norwegian mathematician Selberg and Hungarian mathematician Erd\H{o}s in providing an Elementary proof of the well-known \textit{Prime Number Theorem}. In addition to introducing ourselves to the notion of \textit{Arithmetic Functions}, we shall primarily focus our research on obtaining suitable estimates for the \textit{Chebyshev Theta Function} $\vartheta(x)$. Furthermore, we'll try to infer about the asymptotic properties of another function $\rho(x)$, which shall be needed later on in establishing an equivalent statement of our main result. All the mathematical terminologies pertinent to the proof have been discussed in the earlier sections of the text. 
	\texttt{\textbf{2020 MSC : }} \texttt{Primary 11-02 , 11A25 , 11A41 , 11N37 , 11N56} .\\
	\hspace*{60pt}\texttt{Secondary 11N05, 11-03} .
\end{abstract}
\pagenumbering{arabic}

\tableofcontents

\pagestyle{fancy}

\fancyhead[LO,RE]{\markright}
\lfoot[]{Subham De}
\rfoot[]{IIT Delhi, India}

\section{Introduction}

	The \textit{Prime Number Theorem} is arguably one of the most significant results in Analytic Number Theory. It states that the number of primes less than a number \textit{x} approaches $\frac{x}{log(x)}$ as $\textit{x}\rightarrow \infty$, i.e. , 
\begin{center}
$\pi(x)\sim\frac{\textit{x}}{\textit{log(x)}}$
\end{center}
	 First introduced by Legendre in 1798, Chebychev made important progress towards solving this problem in 1852 when he showed that $\pi(x)$ was the same order of magnitude as $\frac{x}{log(x)}$ by using elementary techniques. Although the bounds of $\pi(x)$ were improved later on, the elementary methods employed were largely ad hoc and gave little hope of actually settling the problem.\\
	The first breakthrough came when the theorem was first proven in 1896 by Jaques Hadamard and Charles Jean de la Vall$\acute{e}$e-Poussin using integral theory and the riemann zeta function $\zeta(s)$ defined by,
	\begin{center}
		$ \zeta(s)=\sum\limits_{n=1}^{\infty}\frac{1}{n^s} $
	\end{center}
	The proof was concluded using a trigonometric identity. Later simplified proofs were developed using similar techniques.\\
	In the year $1949$, to the utter astonishment of the mathematical world, Erd\H{o}s announced that he and Selberg had developed a completely elementary proof of the theorem. The proof described explicitly in the paper gives an impression of what Selberg mentioned in his proof.

\section{The Prime Number Theorem}
\subsection{Arithmetic Functions}
	We first introduce some definitions of important tools required for the statement and the proof:
	\begin{dfn}\label{def1}
		 For each $\textit{x}\geq0$ ,we define,
		 \begin{center}
		 	 $ \pi(x)=$The number of primes $\leq \textit{x}$.
		 	\end{center}
	 	\end{dfn}
	\begin{dfn}\label{def2}
		For each $\textit{x}\geq0$ , we define,
		\begin{center}
		 $\psi(x)=\sum\limits_{n\le x}\Lambda(n)$ ,
		 \end{center}
		 Where ,  
		 \begin{eqnarray}
		  \Lambda(n) 
		 =\left\{
		 \begin{array}{cc}
		 log(p)\mbox{ }, &\mbox{ if }n=p^m,\mbox{ } p^m\leq x,\mbox{ }m\in \mathbb{N}\\
		 0\mbox{ }, &\mbox{ otherwise }.
		 \end{array}
		 \right.
		 \end{eqnarray} \\
		 $ \Lambda(n)$ is said to be the "\textit{Mangoldt Function}"  .\\
		Therefore,
		\begin{eqnarray}
		\psi(x) = \sum\limits_{n\leq x}\Lambda(n) = \sum\limits_{m=1}^{\infty}\sum\limits_{p ,  p^m\leq x} \Lambda(p^m) = \sum\limits_{m=1}^{\infty}\sum\limits_{p\leq x^{\frac{1}{m}}} log(p)
		\end{eqnarray}
	\end{dfn} 
	\begin{dfn}
		For each $x\geq0$, we define,
		\begin{center}
			$ \vartheta(x) = \sum\limits_{p\leq x} log(p) $.
		\end{center}  
	\end{dfn} 
\begin{rmk}\label{rmk1}
		From the two definitions, it can be deduced that,
	\begin{center}
		$ \psi(x) = \sum\limits_{m\leq log_{2}x} \vartheta(x^{\frac{1}{m}}) $
	\end{center} 
\end{rmk}

	\begin{dfn}\label{def3}
		(M\"{o}bius Function)  The M\"{o}bius Function $ \mu $ is defined as follows :
		\begin{center}
			 $\mu(1) = 1$.
		\end{center}
		
		If $ n>1 $, such that suppose, $ n = {{p_{1}}^{a_{1}}}{{p_{2}}^{a_{2}}}{{p_{3}}^{a_{3}}}\cdots{{p_{k}}^{a_{k}}} $ . Then,
		\begin{eqnarray}
		\mu(n) =\left\{
		\begin{array}{cc}
		(-1)^{k}\mbox{ }, &\mbox{ if }a_{1} = a_{2} = \cdots\cdots = a_{k} = 1\\
		0\mbox{ }, &\mbox{ otherwise }.
		\end{array}
		\right.
		\end{eqnarray}
	\end{dfn} 
\begin{dfn}\label{def4}
	 (Big $O$ Notation)  If $g(x) > 0$ for all $x\geq a$, we write, $f(x) = O(g(x))$ to mean that, the quotient, $\frac{f(x)}{g(x)}$ is bounded for all $x\geq a$;  i.e., there exists a constant $M > 0$ such that,
\begin{center}
$\arrowvert f(x)\arrowvert \leq Mg(x)$, for all $x\geq a$ .
\end{center}
\end{dfn}
	\begin{dfn}\label{def5}
		We say g(x) $\sim$ f(x) if, \space\space  $\lim\limits_{x\to\infty}\frac{f(x)}{g(x)}= 1$.
	\end{dfn} 
	\subsection{Statement of the Theorem}
	\begin{thm}\label{thm1}
		(Prime Number Theorem)   
		\begin{center}
			$ \pi(x) \sim \frac{x}{log(x)} $
		\end{center}
		In other words, 
		\begin{center}
			$ \lim\limits_{x\to\infty}\frac{\pi(x)log(x)}{x} = 1 $
		\end{center} . 
	\end{thm} 
A priori having all the neccessary definitions above, we intend to prove the \textit{Prime Number Theorem}, with an approach towards obtaining a better perception of the asymptotic behaviour of $\pi(x)$ .
\begin{thm}\label{thm2}
	We have,
	\begin{center}
		$ \vartheta(x) \sim \pi(x)log(x) $
	\end{center}
\end{thm}
	\begin{proof}
		Using definition,
	\begin{center}
			$ \vartheta(x) = \sum\limits_{x^{1-\varepsilon}\leq p\leq x} log(p)\lfloor\frac{log(x)}{log(p)}\rfloor \leq \sum\limits_{p\leq x} log(x) \leq \pi(x)log(x) $\\
			\vspace{10pt} (by Theorem of partial sums of Dirichlet product )
	\end{center}

	Therefore we have, for any  $ \varepsilon>0 $ ,
	\begin{center}
		$ \vartheta(x)\geq \sum\limits_{x^{1-\varepsilon}\leq p\leq x} log(p) \geq \sum\limits_{x^{1-\varepsilon}\leq p\leq x} (1-\varepsilon)log(x) = (1-\varepsilon)(\pi(x) +O(x^{1- \varepsilon}))log(x)  $
	\end{center}

	Since we can choose $ \varepsilon>0  $ arbitrarily, hence, 
	\begin{center}
			$ \Rightarrow \vartheta(x) \rightarrow \pi(x)log(x) $, \space  as $ x \rightarrow \infty $\\
		$\Rightarrow \vartheta(x) \sim \pi(x)log(x) $ .
	\end{center}
\end{proof}
\subsection{Equivalent Approach of proving PNT}
	\textit{Theorem \eqref{thm2}} implies that, in order to prove the \textit{Prime Number Theorem}, it only suffices to prove that, $ \vartheta(x) \sim x $ .
	\begin{thm}\label{thm3}
		(Alternative Statement of Prime Number Theorem)
		\begin{center}
			$ \vartheta(x) \sim x $, \space\space for $x\geq0$.	
		\end{center}  
	\end{thm} 
For the proof, we shall be needing the following identity formulated by \textit{Selberg}.	
\begin{thm}\label{thm4}
	(Selberg's Asymptotic Formula)  For $x > 0$, we have,
\begin{center}
	$\psi(x)log(x) + \sum\limits_{n\leq x}\Lambda(n)\psi(\frac{x}{n}) = 2xlog(x)+ O(x)$ .
\end{center}
\end{thm}

\section{An Estimate for the Chebyshev Theta Function $\vartheta(x)$}
\subsection{Approximations for $\theta_{n}$}
Given $x>0$ to be a positive real number  and $d\in \mathbb{N}$ , we define,
\begin{dfn}\label{def6}
           	$ \lambda_{d} = \lambda_{d,x} = \mu(d)log^{2}(\frac{x}{d}) $
\end{dfn}

\begin{dfn}\label{def7}
	        $ \theta_{n} = \theta_{n,x} = \sum\limits_{d/n} \lambda_{d} $, \\
	        for any $ n\in \mathbb{N}$ .    
\end{dfn}

Using \textit{Definitions \eqref{def6}} and \textit{\eqref{def7}},
\begin{thm}\label{thm5}
We have,
  \begin{center}
	$\theta_{n} =
	\left\{
	\begin{array}{cc}
	log^{2} x,&\mbox{ for } n=1,\\
	log(p),&\mbox{ for } n=p^{\alpha},\alpha\geq 1,\\
	2log(p)log(q),&\mbox{ for } n = p^{\alpha}q^{\beta}, \alpha\geq 1,\beta\geq 1,\\
	0,&\mbox{ otherwise} .
	\end{array}
	\right. $
\end{center}
\end{thm}

\begin{proof}
First case follows from \textit{Definition \eqref{def6}}, since if $ n=1$ , then we have $ d=1 $ , and consequently the result follows.
As for the second case, we first observe that, $ \mu(p^{\alpha})=0$, for $\alpha>1$.\\
Hence ,
\begin{center}
	$ \theta_{n} = log^{2}(x) - log^{2}(\frac{x}{p}) = log^{2}(x) - (log(x)-log(p))^{2} = 2log(x)log(p) -$ \hspace*{270pt} $ log^{2}(p) = log(p)log(\frac{x^{2}}{p}) $.
\end{center}

For the third case, the result follows from the properties of $\mu$ and $log$ . We have,
\begin{center}
$ \theta_{n} = log^{2}(\frac{x}{1}) - log^{2}(\frac{x}{p}) - log^{2}(\frac{x}{q}) + log^{2}(\frac{x}{pq}) = log(p)log(q)  $.
\end{center}
The fourth case can be proved using induction. Consider $n$ to be squarefree , since if $p^{2}$ divides n, then $\mu(d) = 0$ for all such $d$ which is(are) divisible by $p^{2}$ (and consequently higher powers of p) .\\\\
Furthermore, we assume that, $n={p_{1}}{p_{2}}{p_{3}}\cdots{p_{k}}$ .\\
Then,
\begin{center}
 $ \theta_{n,x} = \theta_{n/p_{k},x} -  \theta_{n/p_{k},x/p_{k}} $ \space\space (Using \textit{Definition \eqref{def7}})
 \end{center}
Here we observe that the third case is independent of $x$ and hence the result follows.
\end{proof}
It follows from \textit{Definition \eqref{def7}} and \textit{Theorem \eqref{thm5}} that ,
\begin{eqnarray}\label{1}
\sum\limits_{n\leq x}\theta_{n} = \sum\limits_{n\leq x}\sum\limits_{d/n} \lambda_{d} = \sum\limits_{d\leq x} \lambda_{d}\lfloor\frac{x}{d}\rfloor = x\sum\limits_{d\leq x} \frac{\lambda_{d}}{d} + O(\sum\limits_{d\leq x} |\lambda_{d}|)  \\
 =x\sum\limits_{d\leq x}\frac{\mu(d)}{d}log^{2}(\frac{x}{d}) + O(\sum\limits_{d\leq x}log^{2}\frac{x}{d}) = x\sum\limits_{d\leq x}{\frac{\mu(d)}{d}}log^{2}(\frac{x}{d}) + O(x).
\end{eqnarray}
The same sum will be evaluated using the results :
\begin{eqnarray}\label{2}
\sum\limits_{p\leq x}\frac{log(p)}{x} = log(x) + O(1).
\end{eqnarray}
\begin{eqnarray}\label{3}
\psi(x) = O(x) 
\end{eqnarray}

to approximate the remainders.\\
Therefore,
\begin{center}
$\sum\limits_{n\leq x} \theta_{n} = log^{2}(x) + \sum\limits_{p^{\alpha}\leq x} log(p)log(\frac{x^{2}}{p})  +2\sum\limits_{p^{\alpha}q^{\beta}\leq x,p<q} log(p)log(q) $
\end{center}
\begin{eqnarray}  \label{4}
= \sum\limits_{p\leq x} log^{2}(p) + \sum\limits_{pq\leq x} log(p)log(q) + O(\sum\limits_{p\leq x} log(p)log(\frac{x}{p}) \end{eqnarray}\\
\hspace*{100pt}$+ O(\sum\limits_{p^{\alpha}\leq x, \alpha>1} log^{2}(x)) + O(\sum\limits_{p^{\alpha}q^{\beta}\leq x, \alpha>1} log(p)log(q)) + log^{2}(x)$\\\\
\hspace*{120pt} $ = \sum\limits_{p\leq x} log^{2}(p) + \sum\limits_{pq\leq x} log(p)log(q) + O(x) $.\\\\
(Using the results \textit{\eqref{2}} and \textit{\eqref{3}})\\\\
Therefore, from \textit{\eqref{1}} and \textit{\eqref{4}}, we get,
\begin{eqnarray}\label{8}
\sum\limits_{p\leq x} log^{2}(p) + \sum\limits_{pq\leq x} log(p)log(q) = x\sum\limits_{d\leq x}\frac{\mu(d)}{d}log^{2}{\frac{x}{d}} + O(x) .
\end{eqnarray} 
\subsection{Divisor Function $\tau(n)$}
	The next step is to estimate the expression corresponding to the Right-Hand-Side. For this purpose, we shall use the following two important results :
	\begin{eqnarray}\label{6}
	 \sum\limits_{n\leq x} \frac{1}{n} = log(z) + C_{1} +O(z^{-\frac{1}{4}})  ,\end{eqnarray}
	  and,
	  
\begin{eqnarray}\label{7}
\frac{\tau(n)}{n} = \frac{1}{2} log^{2}(z) + C_{2} log(z) +C_{3} + O(z^{-\frac{1}{4}}) .
\end{eqnarray}
Where, we define,
\begin{dfn}\label{def8}
	For any $n\in N$, 
	\begin{center}
		$\tau(n) := $ The number of divisors of n .
	\end{center}
\end{dfn} 
  Also, in the above relations, each $C_{i} \mbox{\space};i = 1,2,3$ is an absolute constant. Also the result \textit{\eqref{6}} is well-known and we can easily deduce the result \textit{\eqref{7}} by partial summation formula to the well known result :
\begin{center}
	$\sum\limits_{n\leq x} \tau(n) = z log(z) + C_{4} z + O(\sqrt{z})	$ .
\end{center}
Where $C_{4}$ is also another absolute constant.\\
Using \textit{\eqref{6}} and \textit{\eqref{7}}, we obtain the following identity : 
\begin{center}
	$log^{2}(z) = 2\sum\limits_{n\leq x}\frac{\tau(n)}{n} + C_{5}\sum\limits_{n\leq x}\frac{1}{n} + C_{6} + O(z^{-\frac{1}{4}})$ .
\end{center}
Where $C_{5}$ and $C_{6}$ are also absolute constants.
\subsection{Asymptotic Relation for $\vartheta(n)$}
By applying $z = x/d$ in the above identity, we get,\\\\
$\sum\limits_{d\leq x}\frac{\mu(d)}{d}log^{2}(\frac{x}{d}) = 2\sum\limits_{d\leq x}\frac{\mu(d)}{d} \sum\limits_{r\leq x/d}\frac{\tau(r)}{r} + C_{5}\sum\limits_{d\leq x} \frac{\mu(d)}{d} \sum\limits_{r\leq x/d} \frac{1}{n} + C_{6}\sum\limits_{d\leq x}\frac{\mu(d)}{d} + O(x^{-\frac{1}{4}}\sum\limits_{d\leq x}d^{-\frac{3}{4}}) $\\\\
\hspace*{100pt}$ = 2\sum\limits_{dr\leq x}\frac{\mu(d)\tau(r)}{dr}	
 +C_{5}\sum\limits_{dr\leq x}\frac{\mu(d)}{dr} + C_{6}\sum\limits_{d\leq x}\frac{\mu(d)}{d} + O(1) $\\\\
 \hspace*{100pt}$ = 2\sum\limits_{n\leq x}\frac{1}{n}\sum\limits_{d/n} \mu(d)\tau(\frac{n}{d}) + C_{5}\sum\limits_{n\leq x}\frac{1}{n} \sum\limits_{d/n} \mu(d)
  + O(1) $ \\\\
  \hspace*{100pt} $ = 2\sum\limits_{n\leq x} \frac{1}{n} + C_{5} + O(1) = 2log(x) +O(1)$ .\\\\
  ( We used the following two important results :\\
	\hspace*{70pt}$\sum\limits_{d/n} \mu(d)\tau(\frac{n}{d}) = 1,\space \hspace{20pt} and, \hspace{20pt}\sum\limits_{d\leq x}\frac{\mu(d)}{d} = O(1). \hspace{20pt}$)\\\\
	Now Identity \textit{\eqref{8}} implies,
	\begin{eqnarray}\label{9}
	\sum\limits_{p\leq x}{log^{2}(p)} + \sum\limits_{pq\leq x} log(p)log(q) = 2xlog(x) + O(x),
	\end{eqnarray} 
	The following expression,
	\begin{center}
		$\sum\limits_{p\leq x} log^{2}(p) = \vartheta(x) log(x) + O(x) $ .
	\end{center}
	helps us conclude,
	\begin{eqnarray}\label{5}
	\vartheta(x) log(x) + \sum\limits_{p\leq x}log(p) \vartheta(\frac{x}{p}) = 2x log(x) + O(x) .
	\end{eqnarray}
	Applying partial summation on both sides of the identity \textit{\eqref{9}}, we obtain,
	\begin{eqnarray}\label{13}
	\sum\limits_{p\leq x} log(p) + \sum\limits_{pq\leq x} \frac{log(p)log(q)}{log(pq)} = 2x + O\left(\frac{x}{log(x)}\right)  .
	\end{eqnarray}
	This gives,\\\\
	$\sum\limits_{pq\leq x} log(p)log(q) = \sum\limits_{p\leq x}log(p)\sum\limits_{q\leq x/p}log(q) = 2x\sum\limits_{p\leq x}\frac{log(p)}{p} - \sum\limits_{p\leq x} log(p)\sum\limits_{qr\leq x/p} \frac{log(q)log(r)}{log(qr)}$\\\\
	\hspace*{320pt}$ + O\left(x\sum\limits_{p\leq x}\frac{log(p)}{p\left(1 + log\left(\frac{x}{p}\right)\right)}\right)$\\\\
	\hspace*{160pt}$ = 2xlog(x) - \sum\limits_{qr\leq x} \frac{log(q)log(r)}{log(qr)} \vartheta\left(\frac{x}{qr}\right) + O(xlog(log(x)))$.\\\\
	Applying the above results in the identity \textit{\eqref{9}}, we get,
	\begin{eqnarray}\label{10}
	\vartheta(x)log(x) = \sum\limits_{pq\leq x} \frac{log(p)log(q)}{log(pq)} \vartheta\left(\frac{x}{pq}\right) + O(xlog(log(x))).
	\end{eqnarray} 
	\section{Study of the function  $\rho(x)$}
	\subsection{Formal Definition}
	Our aim is to show, $\frac{\vartheta(x)}{x} \rightarrow 1$ as $x\rightarrow\infty$. It is useful to consider,
	\begin{center}
		$\rho(x) := \vartheta(x) - x$\\
		i.e., $\vartheta(x) = x +\rho(x)$
	\end{center}
Hence, we explore the remainder term $\rho(x)$ in the above definition.
\subsection{Approximations for $\rho(x)$}
It implies from \textit{\eqref{5}} that,
\begin{eqnarray}\label{11}
\rho(x)log(x) = - \sum\limits_{p\leq x}log(p)\rho(\frac{x}{p}) + O(x)
\end{eqnarray}

Now, 
\begin{center}
	$\sum\limits_{p\leq x} \frac{log(p)}{x} = log(x) + O(1)$ , 
\end{center}
Applying \textit{partial summation formula} on both sides, 
\begin{center}
	$\sum\limits_{pq\leq x} \frac{log(p)log(q)}{pq} = \frac{1}{2} log^{2}(x) + O(log(x))$ ,
\end{center}
Applying partial summation on both sides again, 
\begin{center}
	$\sum\limits_{pq\leq x} \frac{log(p)log(q)}{pqlog(pq)} = log(x) + O(log(log(x)))$ ,
\end{center}
Using \textit{\eqref{10}}, we get,
\begin{eqnarray}\label{12}
\rho(x)log(x) = \sum\limits_{pq\leq x}\frac{log(p)log(q)}{log(pq)} \rho(\frac{x}{pq}) +O(xlog(log(x)))
\end{eqnarray}
Therefore we conclude from \textit{\eqref{11}} and \textit{\eqref{12}},
\begin{center}
	$2|\rho(x)| log(x)\leq \sum\limits_{p\leq x} log(p)|\rho(\frac{x}{p})| + \sum\limits_{pq\leq x}\frac{log(p)log(q)}{log(pq)}|\rho(\frac{x}{pq})| + O(log(log(x)))$ ,
\end{center}

By partial summation on both the sides,
\begin{center}
	$2|\rho(x)| log(x) \leq \sum\limits_{n\leq x}\{ \sum\limits_{p\leq x} log(P) + \sum\limits_{pq\leq n} \frac{log(p)log(q)}{log(pq)}\}.\{|\rho(\frac{x}{n})| - |\rho(\frac{x}{n+1})|\} + O(log(log(x))) $.
\end{center}
\begin{rmk}
	This inequality above follows from the previous one, because each and every term in the sum on the right-hand-side is strictly positive, and also because we are simply adding more terms to the series by allowing $n$ that cannot be written or expressed as the product of two primes(as mentioned earlier in the proof).
\end{rmk}

Hnce using \textit{\eqref{13}} and performing summation on both sides of the above inequlity, we obtain,\\\\

$2|\rho(x)|log(x)\leq 2\sum\limits_{n\leq x} n\{|\rho(\frac{x}{n})| - |\rho(\frac{x}{n+1})| \} + O(\sum\limits_{n\leq x} \frac{n}{1+log(n)}\{|\rho(\frac{x}{n}) - \rho(\frac{x}{n+1})|\}) $\\
\hspace*{280pt}$+ O(xlog(log(x)))$\\\\
\hspace*{40pt}$ = 2\sum\limits_{n\leq x}|\rho(\frac{x}{n})| + O(\sum\limits_{n\leq x}\frac{n}{1+log(n)}\{\vartheta(\frac{x}{n}) - \vartheta(\frac{x}{n+1})\}) + O(x\sum\limits_{n\leq x} \frac{1}{n(1+log(n))})$\\
\hspace*{280pt} $+ O(xlog(log(x)))$\\\\
\hspace*{80pt}$ = 2\sum\limits_{n\leq x}|\rho(\frac{x}{n})| + O(\sum\limits_{n\leq x}\frac{n}{1+log(n)} \vartheta(\frac{x}{n})) + O(xlog(log(x)))$\\\\
\hspace*{120pt}$ = 2\sum\limits_{n\leq x}|\rho(\frac{x}{n})| +  + O(xlog(log(x))) $.\\\\\\
Therefore, 
\begin{eqnarray}
|\rho(x)|\leq \frac{1}{log(x)}\sum\limits_{n\leq x}|\rho(\frac{x}{n})| + O(x\frac{log(log(x))}{log(x)}) ,
\end{eqnarray}

whch is the property we shall use later while proving the Prime Number Theorem.
\subsection{Properties of $\rho(x)$}

It follows from \textit{\eqref{2}} after applying partial summation on both sides, 
\begin{center}
	$\sum\limits_{n\leq x}\frac{\vartheta(n)}{n^{2}} = log(x) + O(1)$,
\end{center}
Using the definition, $\vartheta(n) = n + \rho(n)$ and the result, $\sum\limits_{n\leq x}\frac{1}{n} = log(n) + O(1)$, we conclude,
\begin{center}
	$ \sum\limits_{n\leq x}\frac{\rho(n)}{n^{2}} = O(1)$.
\end{center}
Thus for chosen $x>4$ and $x'>x$,  $\exists$ a constant $M_{1}$ such that, 
\begin{eqnarray}
\arrowvert\sum\limits_{x\leq n\leq x'} \frac{\rho(n)}{n^{2}}\arrowvert <M_{1}.
\end{eqnarray}

Accordingly, if $\rho(n)$ does not change its sign in between $x$ and $x'$, then $\exists$  $y \in [x,x']$, such that,
\begin{center}
$\inf\limits_{x\leq y\leq x'}\arrowvert\frac{\rho(y)}{y}\arrowvert\{log(\frac{x'}{x}) + O(1)\} =\inf\limits_{x\leq y\leq x'}\arrowvert\frac{\rho(y)}{y}\arrowvert\sum\limits_{x\leq n\leq x'}\frac{1}{n}$ \\
\hspace*{150pt}$ \leq  \sum\limits_{x\leq n\leq x'}\inf\limits_{x\leq y\leq x'}\arrowvert\frac{\rho(y)}{y}\arrowvert\frac{1}{n}$\\
\hspace*{150pt}$\leq \arrowvert\sum\limits_{x\leq n\leq x'}\frac{\rho(n)}{n^{2}}\arrowvert < M_{1}$ .
\end{center}
Thus, for different constant $M_{2}\geq 0$, 
\begin{eqnarray}
\arrowvert\frac{\rho(y)}{y}\arrowvert < \frac{M_{2}}{ log(\frac{x'}{x})} .
\end{eqnarray}
\begin{rmk}
The above holds true even if $\rho(n)$ changes sign also, since there will always be a $y$ in the above interval, such that, $\arrowvert\rho(y)\arrowvert < log(y)$.
\end{rmk}

Therefore, for a fixed $0<\delta<1$ and $x>4$, $\exists$ $y\in[x,xe^{M_{2}/\delta}]$ satisfying,
\begin{eqnarray}\label{1}
\arrowvert\rho(y)\arrowvert < \delta y.
\end{eqnarray}

\textit{\eqref{13}} implies, for $y < y'$,
\begin{center}
	$0\leq \sum\limits_{y\leq p\leq y'}log(p) \leq 2(y' - y) + O\left(\frac{y'}{log(y')}\right)$ , \\
\end{center}
If $y/2 \leq y' \leq 2y, y>4$, by Definition of $\vartheta(x)$, we obtain,
\begin{center}
	$\arrowvert\rho(y') - \rho(y)\arrowvert \leq y'-y + O\left(\frac{y'}{log(y')}\right)$
	\vspace{10pt}
	$\Rightarrow\arrowvert\rho(y')\arrowvert \leq 	\arrowvert\rho(y')\arrowvert + \arrowvert y' - y\arrowvert + \frac{M_{3}y'}{log(y')}$ .
\end{center}
Where $M_{3}$ is a constant whose existence follows from the definition of the big $O$ notation.\\
Considering an interval $(x,xe^{M_{2}/\delta})$ , using \textit{\eqref{1}} , we can assert that, $\exists$ $y$ in the above interval with,
\begin{center}
 $ \arrowvert\rho(y)\arrowvert < \delta y$ .
\end{center}
Therefore, for any $y'\in [y/2,2y]$, \\
\begin{center}
$	\arrowvert\rho(y')\arrowvert \leq \delta y + \arrowvert y' - y \arrowvert + \frac{M_{3}y'}{log(x)}$
\vspace{10pt}
$\Rightarrow\arrowvert\frac{\rho(y')}{y'}\arrowvert < 2\delta + \arrowvert 1 - \frac{y'}{y}\arrowvert +\frac{ M_{3}}{log(x)}$ .\\ 
\end{center}
Hence, if $x > e^{M_{3}/\delta}$ and $e^{-\delta /2}\leq \frac{y'}{y}\leq e^{\delta /2}$, then,\\
\begin{center}
 $ \arrowvert \frac{\rho(y')}{y'} \arrowvert < 2\delta + (e^{\delta /2} - 1) + \delta < 4\delta$ .
\end{center}
In other words, we conclude, for any $x > e^{M_{2}/\delta}$, the interval $(x,xe^{M_{2}/\delta })$ will contain a sub-interval $(y_{1},y_{1}e^{\delta /2})$, with the property ,
\begin{center}
	$ \arrowvert\rho(z)\arrowvert < 4\delta$ ,
\end{center}
for any $z$ in the sub-interval.\\
Now, we proceed to prove the \textit{Prime Number Theorem}. 
\section{Proof of the Main Theorem}

We intend to prove that, 
\begin{center}
	$\vartheta(x) \sim x$,   	i.e,
	\vspace{10pt}
	$\lim\limits_{x\rightarrow\infty} \frac{\vartheta(x)}{x} = 1$ .	
\end{center}
Equivalently,
\begin{center}
	$\lim\limits_{x\rightarrow\infty} \frac{\rho(x)}{x} = 0$ .
\end{center}
We know by definition, $\vartheta(x) = O(x)$ $\Rightarrow$ $\rho(x) = O(x)$ as well. 
\begin{eqnarray}\label{14}
\arrowvert\rho(x)\arrowvert < M_{4}x.
\end{eqnarray}
For some constant $M_{4}$ and $x > 1$.
Suppose we assume that for some $0 < \alpha < 8$,
\begin{eqnarray}\label{15}
\arrowvert\rho(x)\arrowvert < \alpha x .
\end{eqnarray}
is true for all $x > x_{0} > e^{M_{2}/\delta} $.\\
 We set, $\delta = \alpha /8$. It follows from \textit{Section 5} that intervals of the form $(x,xe^{M_{2}/\delta})$ contain an interval $(y,ye^{\delta /2})$ with,
\begin{eqnarray}\label{16}
\arrowvert\rho(z)\arrowvert <\alpha z/2 .
\end{eqnarray}

Using \textit{\eqref{5}} and \textit{\eqref{14}}, we obtain,\\\\
$\arrowvert\rho(x)\arrowvert \leq \frac{1}{log(x)}\sum\limits_{n\leq x}\arrowvert\rho(\frac{x}{n})\arrowvert + O\left(\frac{x}{\sqrt{log(x)}}\right)$\\\\
\hspace*{50pt}$ < M_{4}\frac{x}{log(x)}\sum\limits_{(x/x_{0})\leq n\leq x} \frac{1}{n} + \frac{x}{log(x)}\sum\limits_{n\leq (x/x_{0})} \frac{1}{n}\arrowvert\frac{n}{x}\rho(\frac{x}{n})\arrowvert + O\left(\frac{x}{\sqrt{log(x)}}\right)$ ,\\\\
Putting $\nu = e^{M_{2}/\delta }$, and applying \textit{\eqref{15}} and \textit{\eqref{16}}, \\\\
$\arrowvert\rho(x)\arrowvert < \frac{\alpha x}{log(x)}\sum\limits_{n\leq (x/x_{0})} \frac{1}{n} - \frac{\alpha x}{2log(x)}$ $\sum\limits_{1\leq r\leq (log(x/x_{0})/log(\nu))}\sum\limits_{y_{r}\leq n\leq y_{r}e^{\delta /2}, \nu ^{r-1}<y_{r}\leq \nu ^{r}e^{-\delta /2}} \frac{1}{n}+ O\left(\frac{x}{\sqrt{log(x)}}\right)$\\

\hspace*{50pt}$ = \alpha x - \frac{\alpha x}{2log(x)}\sum\limits_{1\leq r\leq (log(x/x_{0})/log(\nu)} \delta /2 + O\left(\frac{x}{\sqrt{log(x)}}\right) $\\\\
\hspace*{50pt}$ = \alpha x - \frac{\alpha \delta}{4log(\nu)}x + O\left(\frac{x}{\sqrt{log(x)}}\right) $\\\\
\hspace*{50pt}$ = \alpha \left(1-\frac{\alpha ^{2}}{256M_{2}}\right)x + O\left(\frac{x}{\sqrt{log(x)}}\right) < \alpha \left(1-\frac{\alpha ^{2}}{300M_{2}}\right)x $ .\\\\
for any $x > x_{1}$. \\\\
Clearly, if we look at the iteration process,
\begin{center}
	$\alpha_{n+1} = \alpha_{n}\left(1 - \frac{\alpha _{n}^{2}}{300M_{2}}\right)$, 
\end{center}
It is evident that, the sequence $\{\alpha_{n}\}_{n=1}^{\infty}$
converges to zero if, we start the iteration, for instance, with $\alpha = 4$ (It can be verified quite easily that, then $\alpha_{n} < \frac{M_{5}}{\sqrt{n}}$.
Also $\{\alpha_{n}\}_{n=1}^{\infty}$ is monotonic decreasing and only has a fixed point at $\alpha = 0$.
This implies that, \\
\begin{center}
	$\frac{\rho(x)}{x} \rightarrow 0$ .
\end{center}
This completes the proof of the \textit{Prime Number Theorem}.
\vspace{40pt}
\section*{Acknowledgments}
I'll always be grateful to \textbf{Prof. Baskar Balasubramanyam} ( Associate Professor, Department of Mathematics, IISER Pune, India ), whose unconditional support and guidance helped me in understanding the topic and developing interest towards Analytic Number Theory.\\\\

\end{document}